\documentclass[10pt]{amsart}
\usepackage[leqno]{amsmath}

\usepackage{amsmath}
\usepackage{amsthm}
\usepackage{amssymb}
\usepackage{amscd}
\usepackage{amsxtra}     
\usepackage{epsfig}
\usepackage{verbatim}
\usepackage[all]{xy}
\usepackage{graphicx}
\usepackage{multicol}
\usepackage{mathrsfs}
\usepackage{latexsym}

\theoremstyle{plain} 

\newtheorem{thm}{Theorem}[section]
\newtheorem{lem}[thm]{Lemma}
\newtheorem{prop}[thm]{Proposition}

\newtheorem{cor}[thm]{Corollary}

\newtheorem*{thmA}{Theorem A}
\newtheorem*{thmB}{Theorem B}

\theoremstyle{definition}
\newtheorem{dfn}[thm]{Definition}
\newtheorem{construct}[thm]{Construction}
\newtheorem{eg}[thm]{Example}

\newtheorem{rmk}[thm]{Remark}

\numberwithin{equation}{section}

\newcommand{\tensor}{\otimes}

 \DeclareMathOperator{\Tor}{Tor}
 \DeclareMathOperator{\Ext}{Ext}
 \DeclareMathOperator{\Hom}{Hom}

\DeclareMathOperator{\add}{add}

 \DeclareMathOperator{\Se}{S}

 \DeclareMathOperator{\pd}{pd}

 \DeclareMathOperator{\depth}{depth}

 \DeclareMathOperator{\syz}{syz}

  \DeclareMathOperator{\modu}{mod}
 \DeclareMathOperator{\MCM}{MCM}
  \DeclareMathOperator{\MF}{MF}   
      \DeclareMathOperator{\gl}{gl.dim}

\newcommand{\Ann}{\textup{Ann}}
\newcommand{\ses}[3]{0 \to {#1} \to {#2} \to {#3} \to 0}

 \def\depth{\text{\rm depth}\,}

 \def\<{\langle}
 \def\>{\rangle}

 \newcommand{\bbar}[1]{\setbox0=\hbox{$#1$}\dimen0=.2\ht0 \kern\dimen0 \overline{\kern-\dimen0 #1}}

 \DeclareMathOperator{\coker}{coker}

 \DeclareMathOperator{\End}{\ensuremath{\mathcal{E}\kern-.125em\mathpzc{nd}}}

 \DeclareMathOperator{\id}{id}

 \newcommand{\m}{\mathfrak{m}}

 \DeclareMathOperator{\Proj}{\mathcal{P}\kern-.125em\mathpzc{roj}}

 \DeclareMathOperator{\rk}{rk}

 \DeclareMathOperator{\spec}{Spec}

 \newcommand{\udot}{\ensuremath{{\lower .183333em \hbox{\LARGE \kern -.05em$\cdot$}}}}

 \newcommand{\ZZ}{\mathbb{Z}}

\begin{document}

\bibliographystyle{plain}

\title[finite global dimension of endomorphism rings]{Vanishing of Ext, cluster tilting modules and finite global dimension of endomorphism rings}

\author{Hailong Dao}
\address{Department of Mathematics\\
University of Kansas\\
 Lawrence, KS 66045-7523 USA}
\email{hdao@math.ku.edu}

\author{Craig Huneke}
\address{Department of Mathematics\\
University of Kansas\\
 Lawrence, KS 66045-7523 USA}
\email{huneke@math.ku.edu}

\thanks{The authors are partially supported by NSF grants DMS 0834050 and DMS 0756853}

\subjclass [2000]{13D07, 16G30, 16G50}

\keywords{Ext vanishing, cluster tilting, Auslander correspondence, non-commutative crepant resolutions}

\maketitle

\begin{abstract}
Let $R$ be a Cohen-Macaulay ring and $M$ a maximal Cohen-Macaulay $R$-module. Inspired by recent striking work by Iyama, Burban-Iyama-Keller-Reiten and Van den Bergh we study the question of when 
the endomorphism ring of $M$ has finite global dimension via certain conditions about vanishing of $\Ext$ modules. 
We are able to strengthen certain results by  Iyama on connections between a higher dimension version of Auslander correspondence and existence of non-commutative crepant resolutions. We also recover and extend to positive characteristics a recent Theorem by  Burban-Iyama-Keller-Reiten on cluster-tilting objects in the category of maximal Cohen-Macaulay modules over  reduced $1$-dimensional hypersurfaces. 
\end{abstract}

\section{Introduction}

Let $R$ be a Cohen-Macaulay ring and $\MCM(R)$ be the category of maximal Cohen-Macaulay modules over $R$. Let us recall:
\begin{dfn}\label{defRe}
An $R$-module $M \in \MCM(R)$ is called \textit{cluster tilting} if: 
$$\add(M) = \{X\in \MCM(R)  | \Ext_R^1(X,M)=0\} = \{X\in \MCM(R) | \Ext_R^1(M,X)=0\} $$

\end{dfn}

Here $\add(M)$ consists of direct summands of direct sums of copies of $M$. 

The concept of cluster-tilting objects in the context of maximal Cohen-Macaulay modules and generalized versions  come up     in the  study of  higher-dimensional analogues of Auslander-Reiten theory for Artin algebra. The existence of these objects have remarkable ramifications on both the singularity of $\spec(R)$ and the representation theory of $R$, in a manner similar to the famous McKay correspondence. As a consequence, cluster-tilting modules  have been studied intensively in recent years. 

It is not easy to understand cluster-tilting modules concretely, even when $R$ is relatively simple. A striking result  was recently obtained by Burban, Iyama, Keeler and Reiten in \cite{BIKR}. Their main Theorem states: 

\begin{thm}(\cite[Theorem 1.5]{BIKR})\label{BIKR}
Let $k$ be an algebraically closed field of characteristic $0$ and $R=k[[x,y]]/(f)$ be a one-dimensional reduced hypersurface singularity. Then $\MCM(R)$ has a cluster-tilting object if and only if $f= f_1\cdots f_n$ such that $f_i\notin (x,y)^2$ for each $i$. 
\end{thm}

The proof of \ref{BIKR} used sophisticated methods such as Auslander-Reiten sequences and a subtle connection to crepant resolutions of $\spec(R)$ due to a combination of results by Iyama, Katz and  Van den Bergh. Our initial goal for this project was to understand how one can obtain Theorem \ref{BIKR} from a pure homological approach, with a view towards proving it  in more general situations. We manage  to provide relatively simple, self-contained proofs of the above Theorem and  other related results on the existence of  cluster-tilting objects in \cite{BIKR},\cite{I1},\cite{I2}. Our proofs are rather direct (in particular, we do not use birational geometry or Auslander-Reiten theory), and they occasionally give more general results. To be precise, we prove the following, see Theorems \ref{mainHyper1}, \ref{mainHyper2}:

\begin{thmA}
Let $k$ be an algebraically closed field of characteristic not $2$ and $R=k[[x,y]]/(f)$ be a one-dimensional reduced hypersurface singularity. Then $\MCM(R)$ has a cluster-tilting object if and only if $f= f_1\cdots f_n$ such that $f_i\notin (x,y)^2$ for each $i$. 
\end{thmA}

In fact, for one direction, we only need to assume $R$ is a local complete reduced hypersurface of dimension $1$ (Theorem \ref{mainHyper2}).

Our approach also leads  to modest extensions, in the commutative case, of very interesting results by Iyama  \cite{I1,I2} on the connection between certain maximal orthogonal categories and endomorphism rings of finite global dimensions. For example, our following result (see \ref{mainGlobal}) describes conditions on vanishing of $\Ext$ modules which determine when an endomorphism ring  has finite global dimension  over a Cohen-Macaulay ring.

\begin{thmB}
Let $R$ be a   Cohen-Macaulay ring of dimension $d\geq 3$. Let $M \in \MCM(R)$ such that $M$ has a free summand  and $A=\Hom_R(M,M)$ is $\MCM$. For an integer $n>0$ let:
$$ M^{\perp_n} =  \{X\in \MCM(R) | \Ext_R^i(M,X)=0 \ \text{for} \ 1\leq i \leq n\}$$
Consider the following:
\begin{enumerate}
\item  There exist an integer $n$ such that $1\leq n\leq d-2$ and  $ M^{\perp_n} = \add(M) $.
\item $\gl A\leq d$.
\item $\gl A=d$.
\item $\Ext_R^i(M,M)=0 \ \text{for} \ 1\leq i \leq d-2$ (i.e., $M\in  M^{\perp_{d-2}}$).
\item $ M^{\perp_{d-2}} = \add(M)$.  
\end{enumerate}

Then $(5)\Rightarrow (1) \Rightarrow (2) \Leftrightarrow (3)$. If in addition $R$ is locally Gorenstein on the non-maximal primes of $\spec R$ then $(3)+(4)  \Rightarrow (5)$. If $R$ is also locally regular on the non-maximal primes of $\spec R$ (i.e. $\spec R$ has isolated singularities) then $(3) \Rightarrow (5)$. 
\end{thmB}

One can use the above Theorem to recover  Theorem 5.2.1  in \cite{I2}  which describes similar conditions for modules over isolated singularities.  

Even though the techniques we use provide certain improvements, we would like to stress that the results and methods in the original papers are very insightful, and without the inspiration from those works this project  would not have been conceived.

We now describe the content of our paper. In Section \ref{notation} we gather standard results and notations to be used throughout. Section \ref{globalDim} is devoted to proving \ref{mainGlobal}, a key result on cluster-tilting type conditions and finiteness of global dimension of endomorphism rings. In Section \ref{dim1} we study the existence of cluster-tilting modules over dimension $1$, reduced local hypersurfaces. 

\medskip

\section{Notations and preliminary results}\label{notation}
Let $R$ be a local  Cohen-Macaulay ring. Let  $\MCM(R)$ denote the category  of maximal Cohen-Macaulay modules over $R$. 
The stable category of $\MCM$ $R$-modules $\underline{\MCM}(R) $ can be defined as follows: the objects are objects in 
$\MCM(R)$ and morphisms are defined by:
$$\underline{\Hom}_R(M,N) = \Hom_R(M,N)/I $$
where $I$ denotes the set of morphisms which factor through some free $R$-modules.

\subsection{Matrix factorization and $\MCM$ modules over hypersurfaces}
Suppose that $R=S/(f)$ with $S$ a regular local ring.  A matrix factorization of $f$ is a pair of homomorphisms between free modules $\varphi: F\to G$ and $\psi: G\to F$ such that $\varphi\psi = f\id_G$ and $\psi\varphi = f\id_F$. A morphism between two factorizations is a commutative diagram: 

\[
\xymatrix{
F\ar[r]^{\varphi}  \ar[d]^{\alpha}  & G  \ar[r]^{\psi}  \ar[d]^{\beta}  &F \ar[d]^{\alpha} \\
F'\ar[r]^{\varphi'}   & G'  \ar[r]^{\psi'}   &F' }
\]

It follows that the matrix factorizations of $f$ form a category denoted by $\MF_S(f)$. If we identify morphisms which are homotopic (see, for example \cite{Yo}) we get the stable category $\underline {\MF}_S(f)$. 
The functor which takes a matrix factorization $(\varphi, \psi)$ to $\coker \varphi$ induces a map  of categories between $\MF_S(f) $and $\MCM(R)$ and an equivalence of categories between $\underline {\MF}_S(f)$ and $\underline{\MCM}(R) $.

\subsection{Kn\"{o}rrer periodicity} We now assume that $S=k[[x_1,\cdots,x_n]]$ and $R=S/(f)$ with $k$ an algebraically closed field of characteristic not equal to $2$. The important result below is due to Kn\"{o}rrer in \cite{K}, (see  also \cite{So} for related results in characteristic $2$ case): 

\begin{thm}\label{Kno}
The functor:
$$H: \MF_S(f) \to \MF_{S[[u,v]]}(f+uv)$$
that associates $(\varphi, \psi) \in \MF_S(f)$ to the factorization given by:

\[\bigg(
\begin{bmatrix}
u  & \psi \\
\varphi  &{-v}
\end{bmatrix},
\begin{bmatrix}
v & \psi \\
\varphi &-u
\end{bmatrix}
\bigg)
\]
induces an equivalence of categories between  $F: \underline{\MCM}(R) \simeq \underline{\MCM}(R^{\sharp}) $ with $R^{\sharp}=S[[u,v]]/(f+uv)$.
\end{thm} 

The next Proposition lists a couple of convenient facts about the functor $F$ we shall need. 

\begin{prop}
We use the same setting of Theorem \ref{Kno}. For any two modules $M,N \in \MCM(R)$ we have:
\begin{enumerate}
\item $F(\syz_1^R(M)) \cong \syz_1^{R^{\sharp}}(F(M))$.
\item $\Ext_R^i(M,N)=0$ if and only if $\Ext_{R^{\sharp}}^i(F(M),F(N))=0$. 
\end{enumerate}
\end{prop}

\begin{proof}
$(1)$ follows from the fact that the first syzygy of $M$ corresponds to the same pair of matrices but with reverse order. $(2)$ is a combination of $(1)$ and the fact that:
$$\Ext_R^1(M,N) \cong  \underline{\Hom}_R(\syz_1^R(M),N)$$ 
 
\end{proof}

\subsection{Pushfowards and Serre's conditions $(\Se_n)$}
Let $R$ be a commutative Noetherian ring and $M$ a finite $R$-module. 
For a non-negative integer $n$, $M$ is said to satisfy $(\Se_n)$ if:
$$ \depth_{R_p}M_p \geq \min\{n,\dim(R_p)\} \ \forall p\in \spec(R)$$

Now, suppose that $R$  is locally Gorenstein in codimension $1$. 
Let  $M$ be a torsion-free (equivalent to
$(\Se_1)$) $R$-module. Consider a short exact sequence : $$0 \to W
\to R^{\lambda} \to M^* \to 0 $$ Here $\lambda $ is the minimal
number of generators for $M^*$. Dualizing this short exact sequence
and noting that $M$ embeds into $M^{**}$ we get an exact sequence:
$$ 0 \to M \to R^{\lambda} \to M_1 \to 0$$
This exact sequence is called the \textit{pushforward} of $M$. The following result is modified from \cite[1.6]{HJW}:

\begin{prop}\label{pushforward}
Let $R$ be a  Cohen-Macaulay ring which is Gorenstein in codimension one with $\dim R =d\geq 2$. Let $M,M_1$ be as above. Then for any $p\in \spec(R)$:

\begin{enumerate}
\item $M_p$ is free if and only if $(M_1)_p$ is free. 
\item  If $R_p$ is Gorenstein  and $M_p$ is a maximal Cohen-Macaulay $R_p$-module, then so is $(M_1)_p$.
\item $\depth_{R_p}(M_1)_p \geq \depth_{R_p}M_p -1$.
\item  If $M$ satisfies $(\Se_k)$, then $M_1$ satisfies
$(\Se_{k-1})$ for any $1 \leq k \leq d$.
\end{enumerate}
\end{prop}

\begin{proof}
(1) is easy. For (2) as  $R_p$ is Gorenstein, it follows that  $M_p^*\in \MCM(R_p)$ and dualizing:
$$\ses{W_p}{F_p}{M_p^*} $$
gives $(M_1)_p \cong W^*_p$, which is in $\MCM(R_p)$. 
Assertion (3) follows by counting depths, and (4) is a combination of (2) and (3). Note that since $d\geq 2$, one can use pushforward even when $k=1$. 
\end{proof}

\subsection{Endomorphism rings and global dimensions}\label{Gldim}
Let $R$ be a Noetherian commutative ring and $A$ be an (not necessarily commutative)  Noetherian $R$-algebra. The {\it global dimension} of $A$, $\gl (A)$,  is the supremum of the set of projective dimensions of all (right) A-modules. Since $A$ is Noetherian, one can also define  $\gl(A)$ using left $A$-modules (\cite[Chapter 7, 1.11]{McRo}). 

Throughout the paper we will be interested only in the case $A=\Hom_R(M,M)$, for $M$ a finitely generated $R$-module. Consequently, $A$ is a Noetherian $R$-algebra. Let $N$ be another $R$-module. Then $\Hom_R(M,N)$ has a natural structure as a right $A$-module given as follows. For $t\in A$ and $f\in \Hom_R(M,N)$, define $ft(a) = f(t(a))$ for $a \in M$. 

We note that sometimes in the literature one works over $A^{op}$, the opposite ring of $A$. It makes no difference in our situation, since  $$\gl (A) = \gl(A^{op})$$ as left $A^{op}$-modules can be identified  naturally with right $A$-modules. Also note that $A$ and $A^{op}$ are naturally isomorphic as $R$-modules. For more generalities on endomorphism rings, we refer to the excellent books \cite{ARS} (especially II.1, II.2) and \cite{McRo}. 

\section{$\Ext$ vanishing and global dimension of endomorphism rings}\label{globalDim}
In this section we prove connections between certain $\Ext$ vanishing conditions and endomorphism rings of finite global dimension.  To motivate such study, we recall a  definition by Van den Bergh (see \cite[4.1]{V1}):

\begin{dfn}\label{NCCR}
Let $R$ be a  Cohen-Macaulay normal domain. Suppose that there exists a reflexive module $M$ satisfying:
\begin{enumerate}
\item $A = \Hom_R(M,M)$ is maximal Cohen-Macaulay $R$-module.
\item $A$ has finite global dimension equal to $d=\dim R$.
\end{enumerate} 
Then $A$ is called a non-commutative crepant resolution (henceforth NCCR) of $R$. If $A$ satisfies condition (2) only, it is called a non-commutative desingularization of $R$. 
\end{dfn}

Our starting point is a very interesting result discovered by Iyama in his study of higher-dimensions versions of Auslander correspondence between finite representation type and finite global dimension of additive generator for Artin algebras: 

\begin{thm}(Iyama, \cite[5.2.1]{I2})\label{Iyama}
Let $R$ be a  Cohen-Macaulay  ring with isolated singularities  of dimension $d\geq 3$ and such that $R$ has a canonical module $\omega$. Let $M\in \MCM(R)$ and $A=\Hom_R(M,M)$. The following are equivalent:
\begin{enumerate} 
\item $R,\omega \in \add(M)$, $A\in \MCM(R)$ and $\gl A =d$. 
\item  \begin{eqnarray*} \add(M)= \{X\in \MCM(R) | \Ext_R^i(M,X)=0 \ \text{for} \ 1\leq i \leq d-2\} \\
                          = \{X\in \MCM(R) | \Ext_R^i(X,M)=0 \ \text{for} \ 1\leq i \leq d-2\} 
\end{eqnarray*}

\end{enumerate}

\end{thm}

Iyama's Theorem gives a particularly nice understanding of non-commutative crepant resolutions when $d=3$. However, in higher dimension, the isolated singularity assumption is more restrictive and  condition (2) already rules out a big class of rings for applications. Namely, if  $R$ is a local complete intersection, and $M\in \MCM(R)$, then $\Ext_R^2(M,M)=0$ forces $M$ to be free, since one can complete and ``lift" $M$ to a regular local ring, see \cite{ADS}. With that in mind we shall seek minimal necessary and sufficient conditions for $\Hom_R(M,M)$ to have finite global dimensions when $M$ is a $\MCM$ module over a reasonable Cohen-Macaulay ring $R$. We provide such result in Theorem \ref{mainGlobal}, see also example \ref{BL}.  We can use our result to recover Theorem \ref{Iyama}, whose proof will appear near the end of this section.

We first record some useful results:
\begin{lem}(Lemma 2.3, \cite{Da3})\label{useful}
Let $R$ be a Cohen-Macaulay  ring. Let  $M,N$ be finitely generated $R$-modules and $n>1$ an integer. Consider the two conditions:
\begin{enumerate}
\item  $\Hom(M,N)$  is $(\Se_{n+1})$.
\item  $\Ext_R^i(M,N)=0$ for $1 \leq i \leq n-1$.
\end{enumerate}
If  $M$ is locally free in codimension $n$ and $N$ satisfies $(\Se_n)$, then (1) implies (2). If  $N$ satisfies $(\Se_{n+1})$, then (2) implies (1).
\end{lem}

\begin{proof}
The statement in \cite{Da3} assumes $R$ is local. Since one can localize at each maximal prime without affecting relevant issues, the conclusion is clear.
\end{proof}

\begin{rmk}\label{rmkHom}
Let $M$ be a finite $R$-module and $A=\Hom_R(M,M)$. It is well known that there is an equivalence between the categories of modules in $\add(M)$  and projective (right) modules over $A$ via $\Hom_R(M,-)$ (see for example Lemma 4.12 in \cite{BD} or \cite{Leu}). It follows that any finite right $A$-module $N$ fits into an exact sequence  $$  0 \to \Hom_R(M,N_1) \to \Hom_R(M,P_1) \to \Hom_R(M,P_0) \to N \to 0$$  such that $N_1 \to P_1 \to P_0$ is exact. In particular, if $\dim R\geq 2$ and $M\in \MCM(R)$: 
$$\gl(A) \leq  \sup \{\pd_A\Hom_R(M,N) | N\in \modu(R) \  \text{satisfying} (\Se_2) \} + 2.$$
\end{rmk}

\begin{construct}\label{cons}
The above discussion shows that when investigating projective resolutions of $A$-modules if suffices to consider modules of the
form $\Hom_R(M,N)$. If $R$ is a direct summand
of $M$, one can build a resolution in a particularly nice way. First pick a set of generators $f_1,\cdots,f_n$ of $\Hom_R(M,N)$ which includes a  set of generators of $\Hom_R(R,N)$. Let $\phi$ be the map $M^n \to N$ which takes $(m_1,\cdots,m_n)$ to $f_1(m_1) +\cdots + f_n(m_n)$. Clearly $\phi$ is surjective and $\Hom_R(M,\phi): A^{\oplus n} \to \Hom_R(M,N)$ is also surjective. In other words, one has the short exact sequences:

$$ 0 \to N_1 \to M^{\oplus n} \to N \to 0$$
and
$$0 \to \Hom_R(M,N_1) \to A^{\oplus n} \to \Hom_R(M,N) \to 0$$
Continuing in this fashion one could build an exact complex:  
$$\mathcal{F}:  \cdots \to M^{ n_{i+1}} \to M^{n_i} \to \cdots \to M^ {n_0} \to N \to 0$$
such that $\Hom_R(M,\mathcal{F})$ is an $A$- projective resolution of $\Hom_R(M,N)$.
\end{construct} 

We recall the following notations taken from \cite{I1}.
For a $\MCM$ module over a Cohen-Macaulay ring $R$ and an integer $n>0$, we  denote: 
$$ M^{\perp_n} =  \{X\in \MCM(R) | \Ext_R^i(M,X)=0 \ \text{for} \ 1\leq i \leq n\}$$
and 
$$ ^{\perp_n}M =  \{X\in \MCM(R) | \Ext_R^i(X,M)=0 \ \text{for} \ 1\leq i \leq n\}$$

\begin{thm}\label{mainGlobal}
Let $R$ be a   Cohen-Macaulay ring of dimension $d\geq 3$. Let $M \in \MCM(R)$ such that $M$ has a free summand  and $A=\Hom_R(M,M)$ is $\MCM$. Consider the following:  
\begin{enumerate}
\item  There exist an integer $n$ such that $1\leq n\leq d-2$ and  $ M^{\perp_n} = \add(M) $.
\item $\gl A\leq d$.
\item $\gl A=d$.
\item $\Ext_R^i(M,M)=0 \ \text{for} \ 1\leq i \leq d-2$ (i.e., $M\in  M^{\perp_{d-2}}$).
\item $ M^{\perp_{d-2}} = \add(M)$.  
\end{enumerate}

Then $(5)\Rightarrow (1) \Rightarrow (2) \Leftrightarrow (3)$. If in addition $R$ is locally Gorenstein on the non-maximal primes of $\spec R$ then $(3)+(4)  \Rightarrow (5)$. If $R$ is also locally regular on the non-maximal primes of $\spec R$ (i.e. $\spec R$ has isolated singularities) then $(3) \Rightarrow (5)$. 
\end{thm}

\begin{proof} $ (5)\Rightarrow (1) $ is trivial. 
$(1) \Rightarrow (2)$: By remark \ref{rmkHom} it is enough to prove $\pd_AB\leq d-2$ for $B=\Hom_R(M,N)$ for any $R$-module $N$ satisfying  $(\Se_2)$. We apply Construction \ref{cons} to get a long exact sequence:
$$\mathcal{F}:  0 \to N_{d-2} \to M^{n_{d-3}} \to \cdots \to M^ {n_0} \to N \to 0$$
such that $\Hom_R(M,\mathcal{F})$ is exact. Let $N_j$ be the kernel at the $j-1$ spot. From the construction  we get
$\Ext_R^i(M,N_{j})$ embeds in $\Ext^1(M,M)^{n_{j-1}} = 0$. Also, since $M\in M^{\perp_n}$ one can conclude from the long exact sequence for $\Ext$ at each short exact sequences of Construction \ref{cons} that $N_{d-2} \in  M^{\perp_n} = \add(M)$ by assumption. It  then follows that $\Hom_R(M,N_{d-2})$ is $A$-projective, and the desired conclusion follows. 

$(2) \Leftrightarrow  (3)$: By starting with the $A$-module $B=\Hom_R(M,R/\m)$  for some maximal ideal $\m$ and counting depth along a
(localized at $\m$) $A$-projective resolution of $B$, one can see that $\pd_AB\geq d$.  Thus if $\gl A\leq d$, it is equal to $d$. 

$(3)+(4) \Rightarrow (5)$, assuming $R$ is locally Gorenstein on the non-maximal primes: Let $N \in \MCM(R)$ such that  $\Ext_R^i(M,N)=0$ for $ 1\leq i \leq d-2$.  It suffices to show that $\Hom_R(M,N)$ is a projective $A$-module (the other inclusion is guaranteed by $(4)$. Since $R$ is  a summand of $M$ we can use the pushforward
(see Proposition \ref{pushforward}) to build  a short exact sequence below, with $N_1$ satisfying $\Se_{d-1}$:
$$ 0 \to N \to M^n \to N_1 \to 0 $$
Since $\Ext^1_R(M,N)=0$ one gets 
$$ 0 \to \Hom_R(M,N) \to A^n \to \Hom_R(M,N_1) \to 0 $$
It follows that $\Hom_R(M,N)$ is a first  $A$-syzygy of $\Hom_R(M,N_1)$. Notice that if $d>3$ then $\Ext_R^i(M,N_1)=0$ for $1\leq i \leq d-3$ because of $(4)$. Repeating the process if necessary until we  get a module $N_{d-2}$ which is $(\Se_2)$ such that    $\Hom_R(M,N)$ is a $(d-2)$th  $A$-syzygy of $\Hom_R(M,N_{d-2})$. We claim that $\Hom_R(M,N_{d-2})$ is a second $A$-syzygy. Since $N_{d-2}$ satisfies $(\Se_2)$ we can again use the pushforward to build an exact sequence 
\[ 0 \to N_{d-2} \to M^a  \xrightarrow{\alpha}  M^b. \]
Applying  $\Hom_R(M,-)$ we get an exact sequence of right $A$-modules:

$$ 0 \to \Hom_R(M,N_{d-2}) \to A^a  \to A^b \to X \to 0$$
here $X$ is the quotient of $A^b$ by the image of $\Hom(M,\alpha)$, which proves our claim. 

In summary, $\Hom_R(M,N)$  is a $d$th   $A$-syzygy. By assumption (3) $\Hom_R(M,N)$ must be $A$-projective, so $N\in \add(M)$, and that is what we need to prove.   

Finally, if $\spec R$ has isolated singularities, then since $A$ is $\MCM$ and Lemma \ref{useful} we have $(4)$ automatically, so the last assertion is clear. 

\end{proof}

\begin{cor}\label{BLVCor}
Let $R$ be a Cohen-Macaulay ring of  dimension $d\geq 3$. Let  $M\in \MCM(R)$ such that $M$ is locally free in codimension $2$, $M$ has a free summand and $A=\Hom_R(M,M)$ is also $\MCM$.  Consider the following:
\begin{enumerate}
\item $\{X\in \MCM(R) | \Ext_R^1(M,X)=0\} =  \add(M)  $
\item $\gl(A) = d$
\end{enumerate}
We have: $(1)$ implies $(2)$. If $d=3$, then $(2)$ implies $(1)$. 
\end{cor}

\begin{cor}\label{omega}
Let $R$ be a   Cohen-Macaulay ring of dimension $d\geq 3$ with isolated singularities and suppose $R$ has a canonical module $\omega_R$. Let $M \in \MCM(R)$ such that $M$ has a free summand  and $A=\Hom_R(M,M)$ is $\MCM$. If $\gl(A)\leq d$ then $\omega_R\in \add(M)$.  
\end{cor}

\begin{proof}
Since $\Ext_R^i(M,\omega_R)=0$ for $i>0$, the conclusion follows directly from Theorem \ref{mainGlobal}. 
\end{proof}

\begin{proof}
Proposition \ref{useful} shows that $\Ext_R^1(M,M)=0$. Thus all the assertions follows from Theorem \ref{mainGlobal}. 
\end{proof}

In the following examples we shall investigate the existence of NCCRs (see Definition \ref{NCCR}) for the only known examples of non-Gorenstein local rings of finite $\MCM$ type. The references needed for these can be found in \cite[16.10,16.12]{Yo} or \cite[Examples 11,12]{Leu}. The rings in both examples have dimension $3$ and isolated singularities. 

\begin{eg}
Let $R=k[[x^2,y^2,z^2,xy,yz,zx]]$. It is known that  $R$ is of finite Cohen-Macaulay type and the indecomposable elements of $\MCM(R)$ up to isomorphisms are $R$, the canonical module $\omega = (x^2,xy,xz)$ and $N=\syz_1^R(\omega)$.  Let $M=R\oplus \omega$. Then $\Hom_R(M,M) = M\oplus M$ (because the class group of $R$ is $\ZZ_2$ and generated by $\omega$) is $\MCM$. Since $N=\syz_1^R(\omega)$ there is a non-split exact sequence:
 $$ \ses{N}{F}{\omega}  $$
which shows that $\Ext_R^1(\omega, N) \neq 0$. This shows that $A=\Hom_R(M,M)$ is an NCCR for $R$. 
\end{eg}

\begin{eg}
Let $R=k[[x,y,z,u,v]]/(xz-y^2,xv-yu,yv-zu)$. Then the indecomposable $\MCM$ modules  up to isomorphisms are $R$, the canonical module $\omega = (u,v)$, $N= \syz_1^R(\omega) = (x,y,u)$, $N'= \syz_2^R(\omega)$ and $L=N^{\vee}$. We shall show that there is no NCCR of the form  $A=\Hom_R(M,M)$ such that $R\in \add(M)$. Suppose, by contradiction, that such a module $M$ exists. By Corollary \ref{omega}  we must have $\omega \in \add(M)$. Also, by Lemma \ref{useful}  $\Ext_R^1(M,M)=0$. One can easily check that $\Ext_R^1(\omega,N), \Ext_R^1(N',R)$ and $\Ext_R^1(L,R)$ are not $0$, so $M$ must be of the form $R^a\oplus \omega^b$ for $a,b\geq 1$. But then one can check that  $\Ext_R^1(M,L)=0$, so by Theorem \ref{mainGlobal} $L\in \add(M)$, contradiction. 
\end{eg}

\begin{eg}\label{BL}
Very recently, Buchweitz, Leuschke and Van den Bergh construct in \cite{BLV} non-commutative crepant resolutions using $\MCM$ modules over  hypersurfaces of the form $R=k[X]/\det(X)$ where $X=(x_{ij})$ is an $n\times n$ matrix of indeterminates for $n\geq 2$. Such rings are regular in codimension $2$, but do not have  isolated singularities unless $n=2$. It would be very interesting to see if one can check condition $(1)$ of Corollary \ref{BLVCor} for the modules given in  \cite{BLV}.  
\end{eg}

When $R$ is Gorenstein and is regular on the non-maximal primes, Theorem \ref{mainGlobal} gives a particularly clean result:

\begin{cor}
Let $R$ be a  Gorenstein  ring with isolated singularities and suppose that $\dim R=d>2$. Let $M \in \MCM(R)$ such that $M$ has a free summand  and $A=\Hom_R(M,M)$ is $MCM$. The following are  equivalent:  
\begin{enumerate}
\item $\add(M) = \{X\in \MCM(R) | \Ext_R^i(M,X)=0 \ \text{for} \ 1\leq i \leq d-2\}$.
\item $\gl A <\infty $. 
\item $\gl A=d$. 

\end{enumerate}

\end{cor}

\begin{proof}
This result is just a combination of \ref{useful} and \ref{mainGlobal}.
\end{proof}

The following lemma is probably well-known, but we cannot locate a suitable reference. As we will need it to recover Iyama's result \ref{Iyama} from main Theorem \ref{mainGlobal}, we sketch a proof. 

\begin{lem}
Let $R$ be a Cohen-Macaulay  ring with a canonical module $\omega$.  For an $R$-module $M$ let  $M^{\vee} = \Hom_R(M,\omega)$. Then for any modules $M,N \in \MCM(R)$ and all $i\geq 0$ we have:
$$ \Ext_R^i(M,N)  \cong \Ext_R^i(N^{\vee},M^{\vee}).$$

\end{lem}

\begin{proof}
We shall use induction on $i$. When $i=0$ the needed isomorphism is given by $f \mapsto \Hom(f, \omega)$ (we utilize the isomorphism $M \cong M^{\vee\vee}$). For $i=1$ one can use the Yoneda definition to construct an isomorphism taking an element

$$ \ses NLM$$ 
of $\Ext_R^1(M,N)$ to 
$$ \ses {M^{\vee}}{L^{\vee}}{N^{\vee}} .$$
For $i>1$ let $M_1$ be a first syzygy of $M$ which is also $\MCM$. By induction hypothesis
$$\Ext_R^i(M,N) \cong \Ext_R^{i-1}(M_1,N) \cong \Ext_R^{i-1}(N^{\vee},M_1^{\vee}) $$
Applying $\Hom_R(N,-)$ to the exact sequence:
$$ \ses {M^{\vee}}{F^{\vee}}{M_1^{\vee}}$$
gives $\Ext_R^{i-1}(N^{\vee},M_1^{\vee}) \cong \Ext_R^{i}(N^{\vee},M^{\vee}) $.
\end{proof}

We can now recover Iyama's Theorem: 

\begin{proof}(of Theorem \ref{Iyama}):

Assume $(1)$. Then by \ref{useful} we know that $\Ext_R^i(M,M)=0 \ \text{for} \ 1\leq i \leq d-2$. Now the first equality of $(2)$ follows by Theorem \ref{mainGlobal}. Note that $A \cong \Hom_R(M^{\vee},M^{\vee})$, hence $M^{\vee}$ also satisfies all conditions of $(1)$. So by the first equality of $(2)$, which we already proved, we have:
$$\add(M^{\vee})= \{Y\in \MCM(R) | \Ext_R^i(M^{\vee},Y)=0 \ \text{for} \ 1\leq i \leq d-2\} .$$

But note that $\Ext_R^i(M^{\vee},Y) \cong \Ext_R^i(Y^{\vee},M)$ we obtain:
$$\add(M^{\vee})= \{Y\in \MCM(R) | \Ext_R^i(Y^{\vee},M)=0 \ \text{for} \ 1\leq i \leq d-2\} .$$ 
Let $X=Y^{\vee}$ we get the second inequality of part $(2)$. 

Now assume $(2)$. Then obviously $R,\omega \in \add(M)$. The fact that $\gl A=d$ follows from Theorem \ref{mainGlobal}.

\end{proof}

We note that when $R$ is a Gorenstein normal domain, the condition $(2)$ and $(3)$ of Theorem \ref{mainGlobal} are equivalent. That result is due to Van den Bergh in \cite[4.2]{V1}, we provide here more details of the proof for the convenience of the readers. 

\begin{prop}(Van den Bergh)\label{VB}
Let $R$ be a Gorenstein normal domain of dimension $d$.  Let $M \in \MCM(R)$ such that $A = \Hom_R(M,M)$ is also $\MCM$. If  $\gl(A) <\infty$ then $\gl(A) =  d$. 
\end{prop}  

\begin{proof}
There is a well known spectral sequence for change of rings:
$$ E^{p,q}_2  = \Ext_A^p(B, \Ext^q_R(A,C)) \Rightarrow_p \Ext^n_R(B,C) $$
for any (left) $A$-module $B$ and $R$-module $C$. Let $C=R$, since $A \in \MCM(R)$ we know that
$\Ext^q_R(A,R)=0$ for $q>0$. So one has an isomorphism $\Ext_A^p(B, A^*) \cong \Ext^p_R(B,R) $ for $p>0$. By \cite[Lemma 5.4]{Aus1} $A^* \cong A$ as $A$-modules. So we have $\Ext_A^p(B, A) \cong \Ext^p_R(B,R) $ for  any $A$-module $B$ and any $p>0$. Therefore $\Ext_A^p(B, A)=0$ for $p>d$. Thus if $\pd_AB<\infty$, it is at most $d$. By starting with the $A$-module $B=\Hom_R(M,R/\m)$  for some maximal ideal $\m$ and count depth along a (localized at $m$) $A$-projective resolution of $B$, it is clear that $\pd_AB \geq d$. We can now conclude that $\gl A =d$.  

\end{proof}

\section{Existence of cluster-tilting objects over dimension 1 reduced hypersurfaces}\label{dim1}

The purpose of this section is to give a pure algebraic proof of Theorem \ref{BIKR} which works even over algebraically closed fields of positive characteristic not equal to $2$. 

We recall the following definition for the reader's convenience:

\begin{dfn}\label{defRep}
Let $\mathcal C$ be either $\MCM(R)$  or $\underline{\MCM}(R)$ and $M \in \mathcal C$. We call $M \in \mathcal C$: 
\begin{itemize}
\item \textit{rigid} if $\Ext_R^1(M,M)=0$.
\item \textit{cluster tilting} if 
$$\add(M) = \{X\in \mathcal C | \Ext_R^1(X,M)=0\} = \{X\in \mathcal C | \Ext_R^1(M,X)=0\} $$
\end{itemize}
\end{dfn}

Let $(S,\m)$ be a  complete regular local ring of dimension $2$. Let $R=S/(f)$ be a reduced hypersurface. Assume $f= f_1\cdots f_n$ is the factorization of $f$ into prime elements. For a subset $I \subset \{1,2,\cdots,n \}$ let $f_I= \prod_{i\in I} f_i$ and $S_I=S/(f_I)$ .

The rest of this section is devoted to extending and at the same time giving a direct proof of Theorem \ref{BIKR}. We aim to prove the following result, which will be achieved by combining Theorems \ref{mainHyper1} and \ref{mainHyper2}:
\begin{thm}
Suppose that $S=k[[x,y]]$ where $k$ is an algebraically closed field of characteristic not $2$. Then the following are equivalent:
\begin{enumerate}
\item For every $i$, $f_i\notin \m^2$. 
\item $\MCM(R)$ admits a cluster tilting object. 

\end{enumerate}
\end{thm}

Our approach to this theorem will enable us to prove the same result over any algebraically closed field of characteristic
not equal to $2$. The proof of the above theorem in \cite{BIKR} is quite ingenious, but complex. It uses many techniques
from Auslander-Reiten theory for Artin algebras, as well as subtle connections between crepant resolutions and
NCCRs (see Definition \ref{NCCR}), by using Kn\"orrer periodicity to lift to a three dimensional hypersurface singularity. We also use Kn\"orrer
periodicity to study the ranks of indecomposable modules over $R$. However, we are able to give, for example, a self-contained
proof that the condition  $f_i\notin \m^2$ implies $\MCM(R)$ admits a cluster tilting object, which holds for $S$ any
two-dimensional regular local ring (in particular, $S$ does not even have to contain a field for this direction).
We hope our approach will lead to new insight into similar problems.
We begin by characterizing the indecomposable rigid objects.

\smallskip

\begin{prop}\label{rigid}
Let $(S,\m)$ be a  complete regular local ring of dimension $2$ such that $S/\m$ is algebraically closed of
characteristic not equal to $2$. Let $R=S/(f)$ be a reduced hypersurface, and  assume that $f= f_1\cdots f_n$ is a
factorization of $f$ into prime elements.
Any indecomposable  rigid object in $\MCM(R)$ or $\underline{\MCM}(R)$ is of the form $S_I$ with $I\subset \{1,2,\cdots,n \}$.
\end{prop}

\begin{proof}
We can assume we are in $\underline{\MCM}(R)$. Let $R'= S[[u,v]]/(uv+f)$. By Kn\"{o}rrer's periodicity result (\ref{Kno})
we have $\underline{\MCM}(R)$ and $\underline{\MCM}(R')$ are equivalent. We claim  that any indecomposable  rigid object in $\underline{\MCM}(R')$ has rank $1$. Since the residue field is infinite we can find $t \in R'$ such that $R_1=R'/(t) \cong k[[x,y,z]]/(x^2+y^2+z^n)$, an $A_n$ type simple singularity. Let $M$ represent an indecomposable  rigid object in $\MCM(R')$. Then $M/tM$ is a MCM module over $R_1$. Suppose $\rk M>1$, then so is $\rk M/tM$ as a module over $R_1$. 

It is well known that in $\MCM(R_1)$, all indecomposable objects have rank $1$
(one can prove this assertion by noting that all indecomposable modules over $k[[z]]/(z^n)$ are of the form $k[[z]]/(z^i)$ for $1\leq i\leq n$,
and then apply Kn\"{o}rrer's periodicity to describe all indecomposable $\MCM$ modules over $R_1$). 
As  $M/tM$ has rank bigger than $1$, it has to be decomposable. In other words, $\Hom_{R_1}(M/tM,M/tM)$ has idempotents.
Since $\Ext^1_{R'}(M,M)=0$ we obtain $\Hom_{R_1}(M/tM,M/tM) \cong \Hom_{R'}(M,M)/t\Hom_{R'}(M,M)$. As $R'$ is complete,
one can lift idempotents to $ \Hom_{R'}(M,M)$, contradicting the assumption that $M$ is indecomposable. So $M$ has to have rank $1$,
as claimed. 

All the rank $1$ MCM modules over $R'$  represent elements of the class group of $R'$, which is easy to
understand by inverting $u$ and using well-known sequences describing the relationship between the
class group of $R'$ and the class group of $R'[\frac{1}{u}]$. In particular, any rank $1$ MCM over $R'$ is isomorphic to one of the
ideals $M =(u,f_I)$ (they are elements of the class group of $R'$).  Using Kn\"{o}rrer's periodicity Theorem \ref{Kno}
which induces a bijection between the set of indecomposable rigid objects  of  $\underline{\MCM}(R)$ and $\underline{\MCM}(R')$,
our assertion is now clear.  \end{proof}

\begin{lem}\label{vanish}
Let $R$ be a local, reduced hypersurface of dimension $1$ and $M,N \in \MCM(R)$. Then the following are equivalent:
\begin{enumerate}
\item $\Ext^1_R(M,N)=0$.
\item $\Ext^1_R(N,M)=0$. 
\item $M\tensor N^* \in \MCM(R)$. 
\item $M^*\tensor N \in \MCM(R)$. 
\item $\Tor_2^R(M^*,N)=0$.
\item $\Tor_2^R(M,N^*)=0$.

\end{enumerate}
\end{lem}

\begin{proof}
The equivalence of $(1)$ and $(3)$ is Theorem 5.9 of \cite{HJ}. It also implies $(2) \Leftrightarrow (4)$.
There is an exact sequence (see \cite{Ha}, \cite{Jo} or \cite{Jor}):
$$ \Tor_2^R(M^1,N) \to \Ext^1_R(M,R)\tensor_R N \to \Ext^1_R(M,N) \to \Tor_1^R(M^1,N) \to 0 $$
Here $M^1$ is the cokernel of $F_1^* \to F_2^*$, where $\cdots \to F_2 \to F_1\to F_0 \to M \to 0$ is a minimal resolution of $M$.  Since $M$ is $\MCM$ and $R$ is a hypersurface, we know that $\Ext_1^R(M,R)=0$ and $M^1$ is 
isomorphic to the first syzygy of $M^*$. It follows that $(1) \Leftrightarrow (5)$.

To finish the proof we claim that for $A,B \in \MCM(R)$, $\Tor_2^R(A,B)=0$ if and only if $A\tensor_RB \in \MCM(R)$. 
This will show that $(5) \Leftrightarrow (4)$ and $(6) \Leftrightarrow (3)$. Let $A_1$ the first syzygy of $A$. Then by periodicity, $A$ is a first syzygy of $A_1$. Tensoring the exact sequence $ 0 \to A \to F \to A_1 \to 0$ with $B$ 
we get 
$$ 0 \to \Tor_1^R( A_1,B)  \to A\tensor_RB \to F\tensor_RB \to A_1\tensor_RB \to 0 $$
Since $\Tor_1^R( A_1,B) = \Tor_2^R(A,B)$ has finite length ($R$ is reduced) one concludes that $\depth A\tensor_RB=1$ if and only if $  \Tor_2^R(A,B)=0$ as required.

\end{proof}

\begin{cor}\label{subsets} Let $(S,\m)$ be a  complete regular local ring of dimension $2$. Let $R=S/(f)$ be a reduced hypersurface. Assume $f= f_1\cdots f_n$ is a factorization of $f$ into prime elements.
Let $I,J\subset \{1,2,\cdots,n \}$. 
\begin{enumerate}
\item $S_I^* \cong S_I$.
\item The first syzygy of $S_I$ is $S_{L}$ with $L$ the complement of $I$. 
\item  $\Ext^1_R(S_I,S_J)=0$ if and only if $I\subset J$ or  $J\subset I$.
\item Let $I,J$ be disjoint subsets of $\{1,2,\cdots,n \}$. Let $N\in \MCM(S_I)$ and $M$ be the first $R$-syzygy of $N$.
Then $\Ext^1_R(M, S_J)=0$. 
\item Let $M,N \in \MCM(S_I)$. We have $\Ext^1_R(M,N) = 0$ if and only if $\Ext^1_{S_I}(M,N)=0$.   
\end{enumerate}
\end{cor}

\begin{proof}
Part (1) and (2) are easy computations. 

Part (3): It is easy to see that $S_I^* \cong S_I$ and $S_I\tensor S_J \in \MCM(R)$ if and only if $I\subset J$ or  $J\subset I$. Then we can use Lemma \ref{vanish}. 

Part (4): By Lemma \ref{vanish} we need to show that $\Tor^R_2(M,S_J)=0$. Let $L$ be the complement of $J$. Then $\Tor^R_2(M,S_J)=\Tor^R_2(N,S_L)$. Again by \ref{vanish} and part (1) this is equivalent to $N/f_LN \in \MCM(R)$. But 
$I\subset L$, so $f_L \in \Ann(N)$, thus $N/f_LN=N$. 

Part (5): We use Yoneda definition for $\Ext$. Since any exact sequence in $\MCM(S_I)$ is also an exact sequence in $\MCM(R)$, one direction is easy. Suppose $\Ext^1_{(S_I)}(M,N)=0$ and $0 \to N \to Q \to M \to 0$ represents an element in  $\Ext^1_R(M,N)$. It is enough to show that $Q \in \MCM(S_I)$. It is easy to see that $(f_I)^2Q =0$. Since $Q$ is a Cohen-Macaulay module over $S$, its annihilator in $S$ must be an unmixed ideal of height $1$. It follows that $f_IQ=0$, so the exact sequence must splits, and we are done.

\end{proof}

Before moving on we recall some notations from \cite{BIKR}. We call an object in $\MCM(R)$ \textit{basic} if any indecomposable direct summand occurs  only once. For a permutation
$\omega$ of the set $\{1,\cdots,n\}$, let $S^{\omega}_i=S/(\prod_{j=1}^i f_{\omega(j)})$ and $S^{\omega}=\oplus_{i=1}^n S^{\omega}_i$. 

\begin{thm}\label{mainHyper1} 
Let $(S,\m)$ be a  equicharacteristic complete regular local ring of dimension $2$ such that $S/\m$ is algebraically closed of
characteristic not equal to $2$. Let $R=S/(f)$ be a reduced hypersurface, and  assume that $f= f_1\cdots f_n$ is a
factorization of $f$ into prime elements.
If $\MCM(R)$ admits a basic cluster tilting object then it has to be of the form $S^{\omega}$ and furthermore $f_i \notin \m^2$ for any $i$.
\end{thm}

\begin{proof}
A cluster tilting object is rigid, so by Proposition \ref{rigid} it has to be a direct sum of modules $S_I$. Now \ref{subsets}
shows that it has to be of the form $S^{\omega}$.

For the second assertion, we may assume that our object is $S^{\omega}$ with $\omega(i)=i$. We shall write $S_i$ for 
$S^{\omega}_i$.

We shall use induction on $n$. For $n=1$, then $S^{\omega}=R$, thus $\Ext^1_R(S^{\omega},M)=0$ for any $M\in \MCM(R)$, and they have to be in $\add(R)$ by definition. So $R$ has to be regular. 

Suppose that we can prove our assertion for some value $n\geq 1$. By part (5) of  \ref{subsets} or Lemma 4.9 of \cite{BIKR} one
can deduce that  $ \oplus_{i=1}^{n-1} S_i$ is a cluster tilting object over $S_{n-1}$. By induction $f_i\notin \m^2$ for $1\leq i\leq n-1$. It remains to show that $f_n\notin \m^2$. Suppose it is not the case. Then we can take a module $N\in \MCM(S/(f_n))$ such that $N$ is not free over $S/(f_n)$. Let $M$ be the first $R$-syzygy of $N$. Then $\Ext^1_R(S_i,M)=0$ by Corollary \ref{subsets}. So $M\in \add(\oplus_{i=1}^{n-1} S_i)$ by the definition of cluster tilting. This forces $N$, being the first syzygy of $M$ by periodicity, to be direct sum of modules 
of the form $S_{J_i}$, with $J_i$ the complement of the subset $\{1,\cdots, i\}$. But since $N$ is a module over 
$S/(f_n)$, it has to be a  free $S/(f_n)$-module, a contradiction.

\end{proof}

\begin{thm}\label{mainHyper2}
Let $(S,\m)$ be a complete regular local ring of dimension $2$.  Suppose
that $R=S/(f_1\cdots f_n)$ with every $f_i \notin \m^2$.  Then any  $S^{\omega}$ is a  basic cluster tilting object in $\MCM(R)$. 
\end{thm}

\begin{proof}
Without loss of generality we can assume $\omega(i) =i$. We shall use induction on $n$. The case $n=1$ is obvious since $R$ is then regular and $S^{\omega}=R$. 
We also note that from Corollary \ref{subsets}, $\Ext_R^1(X,S^{\omega}) = \Ext_R^1(S^{\omega},X)
= 0$ for all $X\in \add(M)$. What remains is to prove that any $X$ with this vanishing property
is in $\add(M)$.

Suppose we have already proved the assertion for up to $n-1$ with some $n>1$.  Let $M\in \MCM(R)$ such that $\Ext_R^1(M,S^{\omega}) = \Ext_R^1(S^{\omega},M)=0$. Let $S^{\omega}=R\oplus S^{\omega'}$. For notational convenience let $h= f_1\cdots f_{n-1}$ and $g=f_n$. By the induction hypothesis, $S^{\omega'}$ 
is a cluster tilting object over $S/(h)$. 

The exact sequence $0 \to hM \to M \to M/hM \to 0$ shows that the module $hM$ is in $\MCM(R)$. But $gh=0$ in $R$, so  $hM$ is in $\MCM(R/g)$. Since $R/(g)$ is regular, $hM$ is a free $R/(g)$-module. Pick a minimal system of generators 
$ha_1,\cdots, ha_l$ for $hM$. 

We claim that the submodule $N=(a_1, \cdots, a_l)M$ of $M$ is $R$-free.  Let $N_1$ be the first $R$-syzygy of $N$. We claim that $N_1=gN_1$. Pick any element of $N_1$ which represents a relation $\sum_1^l r_ia_i=0$ with $r_i\in R$. We first notice that  for each $i$, $r_i \in gR$ since $\sum_1^l r_i(ha_i)=0$ and the $ha_i$s form a basis for the free $R/(g)$-module $hM$.

For each $i$, let $r_i = gr_i'$ with $r_i'\in R$. Let $b= \sum_1^l r_i'a_i$, so that $gb=0$.
The fact that $\Ext_1^R(R/(h),M)=0$ implies that $b \in hM$ (take a free resolution of $R/(h)$ and compute $\Ext$, one easily see that $\Ext_1^R(R/(h),M)= (0:_gM)/hM$). So we know that
$b = \sum_1^l hs_ia_i$. It follows that $\sum_1^l (r_i'-hs_i)a_i=0$, so the vector with components $r_i'-hs_i$ is in $N_1$. Multiply by $g$ (note that $gh=0$) we get $(r_1,\cdots,r_l)\in gN_1$ as claimed. By Nakayama's Lemma we can conclude that $N_1=0$, so $N$ is $R$-free.

Now, look at the exact sequence $$0 \to N \to M \to M/N \to 0 \ \ (*).$$  We claim that $g$ is  a nonzerodivisor on $M/N$. Suppose $a \in M$ such that $ga \in N$. Then we can write $ga = \sum r_ia_i$. It follows that  $ \sum r_i(ha_i) =0$, and as in the last paragraph, $a_i =ga_i'$ for each $i$. Look back to the first equation we can write $g(a- \sum r_i'a_i)=0$. As in the previous paragraph, $a-\sum r_i'a_i \in hM$, hence $a \in N+hM \subseteq N$. 

The fact we just proved implies that $\depth M/N=1$, in other words, $M/N \in \MCM(R)$.
But since $N$ is $R$-free, the exact sequence (*) splits. So $\Ext^1_R(M/N,S^{\omega'})=0$. Note that as $M/N \in \MCM(R/(h))$ the last part of Corollary \ref{subsets} tells us that    
$\Ext^1_{S/(h)}(M/N,S^{\omega'})=0$, so by induction hypothesis, $M/N \in \add(S^{\omega'})$. Since $M = N \oplus M/N$ and $N$ is $R$-free, we are done. 

\end{proof}


\begin{thebibliography}{BroSh}



\bibitem{Aus1} M. Auslander, \emph{Rational Singularities and Almost Split Sequences}, Trans. AMS 293 No. 2, (1986), 511--531. 

\bibitem{ADS} M. Auslander, S. Ding, and \O. Solberg, \emph{Liftings and weak liftings of modules}, J. Algebra 156 
(1993), 273–-317.


\bibitem{ARS} M. Auslander, I. Reiten, S. Smal\o, \emph{Representation Theory of Artin Algebras}, Cambridge Studies
in Advanced Mathematics 36, Cambridge University Press, Cambridge, 1995.

\bibitem{BD} I. Burban, Y. Drozd, \textit{Maximal Cohen-Macaulay modules over surfaces singularities}, Trends in Representations of Algebras and Related Topics (2008),
EMS Publishing House, 101--166. 

\bibitem{BH} W. Bruns, J. Herzog, \emph{Cohen-Macaulay rings}, Cambridge Univ. Press, Cambridge (1996).


\bibitem{BIKR}  I. Burban, O. Iyama,  B. Keller, I. Reiten, \emph{Cluster tilting for one-dimensional hypersurface singularities}, Adv. Math. 217 (2008), no. 6, 2443--2484.


\bibitem{BLV}  R.-O. Buchweitz, G. Leuschke, M. Van den Bergh, \emph{Non-commutative desingularization of determinantal varieties}, available at http://front.math.ucdavis.edu/0911.2659.

\bibitem{Da1} H. Dao, \emph{Decency and rigidity over hypersurfaces}, arXiv math.AC/0611568.

\bibitem{Da2} H. Dao, \emph{Some observations on local and projective hypersurfaces}, Math. Res. Let. 15 (2008), no. 2, 207--219. 

\bibitem{Da3} H. Dao, \emph{Remarks on non-commutative crepant resolutions of complete intersections}, Advances in Math., to appear. 

\bibitem{FOV} H. Flenner, L. O'Carroll, W. Vogel, \emph{Joins and intersections}, Springer Mono. Math., (1999). 

\bibitem{Ha} R. Hartshorne, \emph{Coherent functors}, Adv. in Math. 140 (1994), 44--94.

\bibitem{HJ} C. Huneke, D. Jorgensen, \emph{Symmetry in the vanishing of Ext over Gorenstein rings}, Math. Scand. 93 (2003), 161--184. 

\bibitem{HJW} C. Huneke, R. Wiegand, D. Jorgensen,  \emph{Vanishing theorems for complete
intersections}, J. Algebra {238} (2001), 684--702.

\bibitem{HW1} C. Huneke, R. Wiegand, \emph{Tensor products of modules and the rigidity of Tor},
Math. Ann. 299 (1994), 449--476.



\bibitem{I1} O. Iyama, \emph{Higher dimensional Auslander-Reiten theory on maximal orthogonal subcategories}, Adv. Math. 210 (2007), no. 1, 22–-50.

\bibitem{I2} O. Iyama, \emph{Auslander correspondence}, Adv. Math. 210 (2007), no. 1, 51–-82.

\bibitem{Jo} P. Jothilingam, \emph{A note on grade}, Nagoya Math. J. 59 (1975), 149–152.


\bibitem{Jor} D. Jorgensen, \emph{Finite projective dimension and the vanishing of Ext(M,M)}, Comm. Alg. 36 (2008) no. 12, 4461--4471.

\bibitem{K} H. Kn\"{o}rrer, \emph{Cohen-Macaulay modules on hypersurface singularities. I}, Invent. Math., 
88 (1987), 153–-164.

\bibitem{Leu} G. Leuschke, \emph{Endomorphism rings of finite global dimension}, Canad. J. Math. 59 (2007), pp. 332–-342. 



\bibitem{McRo} J.C. McConnell, J.C. Robson, \emph{Noncommutative Noetherian rings}, Wiley, New York, 2000.



\bibitem{So} \O. Solberg, \emph{Hypersurface singularities of finite Cohen-Macaulay type}, Proc. London Math. Soc. 58 (1989), 258--280.

\bibitem{V1} M. Van den Bergh, \emph{Non-commutative crepant resolutions}, The legacy of Niels Henrik Abel, 
749–770, Springer, Berlin, 2004. 



\bibitem{Yo} Y. Yoshino, \emph{Cohen-Macaulay modules over Cohen-Macaulay rings},
Lond. Math. Soc. Lect Notes {146} (1990).






\end{thebibliography}
\end{document}